\documentclass[12pt]{article}

\usepackage[a4paper,margin=1in]{geometry}
\usepackage[T1]{fontenc}
\usepackage[utf8]{inputenc}
\usepackage{amsmath,amssymb,amsthm,mathtools}
\usepackage[protrusion=true,expansion=false]{microtype}
\usepackage[colorlinks=true,citecolor=blue,linkcolor=blue,urlcolor=blue]{hyperref}

\numberwithin{equation}{section}

\newcommand{\R}{\mathbb{R}}
\newcommand{\N}{\mathbb{N}}
\newcommand{\Pp}{\mathbb{P}}
\newcommand{\Ee}{\mathbb{E}}

\newcommand{\one}{\mathbf 1}

\theoremstyle{plain}
\newtheorem{theorem}{Theorem}[section]
\newtheorem{corollary}[theorem]{Corollary}
\newtheorem{lemma}[theorem]{Lemma}
\newtheorem{proposition}[theorem]{Proposition}
\newtheorem{example}[theorem]{Example}
\newtheorem{examples}[theorem]{Examples}

\theoremstyle{definition}
\newtheorem{assumption}[theorem]{Assumption}

\theoremstyle{remark}
\newtheorem{remark}[theorem]{Remark}

\title{Logarithmic Large Deviations for Heavy-Tailed Randomly Stopped Sums}
\author{Jos\'{e} Miguel Zapata}
\date{}

\begin{document}

\maketitle
\begin{abstract}

We derive logarithmic large-deviation asymptotics for randomly stopped sums with nonnegative summands satisfy a weak form of regular variation.  Our approach is based on a count-scale condition that controls  the logarithmic order of the random number of summands  and does not require relative concentration around its expectation.  
A main advantage of the framework is that it applies to every  renewal count, without moment, regular-variation, or domain-of-attraction assumptions on the interarrival distribution. In particular, it covers infinite-mean interarrival regimes in which relative concentration may fail. It is proven that  the mean renewal count is  an admissible count scale even when relative concentration fails. 
We further apply the results to finite heterogeneous insurance portfolios, obtaining dependence-robust asymptotics for unions and weighted aggregate losses, as well as explicit joint shortfall costs under block independence.



\end{abstract}

\section{Introduction}

Randomly stopped sums of the form
\[
        S(t)=\sum_{k=1}^{N(t)}X_k
\]
arise naturally in insurance risk, queueing, finance, reliability, and
other stochastic systems in which both the number and the sizes of the
observed events are random. When the summands have regularly varying
tails, sufficiently large upper deviations are typically generated by
one exceptionally large summand. This one-big-jump mechanism has its roots in the classical large-deviation theory for sums of heavy-tailed random variables, notably in the work of Nagaev~\cite{Nagaev1979}; see also Mikosch and Nagaev~\cite{MikoschNagaev1998}. It has subsequently motivated an extensive literature on large deviations for heavy-tailed random sums; see, e.g., \cite{KluppelbergMikosch1997,TangSuJiangZhang2001,Lin2008,DenisovFossKorshunov2010}.

A classical contribution is due to Kl\"uppelberg and Mikosch
\cite{KluppelbergMikosch1997}, who obtained, for finite-mean claim
sizes, ratio asymptotics of the form
\begin{equation}
\label{eq:ratioAsymptotics}
        \Pp\left(S(t)-\mu m(t)>u\right)
        \sim
        m(t)\bar F(u),
\end{equation}
 where $\mu=\Ee[X_1]$, and $m(t)=\Ee[N(t)]$. 
A representative line of subsequent precise large-deviation results
\cite{TangSuJiangZhang2001,NgTangYanYang2004,Lin2008,
ChenWhiteYuen2021,LiuFuChen2024,GaoDongLiuYan2025}
uses count assumptions that either include or imply the
relative-concentration property
\begin{equation}
\label{eq:concentration}
        \frac{N(t)}{m(t)}
        \xrightarrow{\mathrm P}
        1.
\end{equation}
Depending on the model, this property is assumed explicitly, follows
from a tail condition on the count, or is obtained from a law of large
numbers for the underlying renewal process. Additional upper-tail
conditions are typically imposed to ensure that exceptionally large
values of \(N(t)\) do not alter the precise ratio asymptotic.

For  renewal counts, condition~\eqref{eq:concentration} holds
when the interarrival times have finite mean, but it may fail in
infinite-mean interarrival regimes.\footnote{Here, by a renewal count we mean
\(N(t)=\max\{n\ge0:T_n\le t\}\), \(t\ge0\), where \(T_0=0\),
\(T_n=\tau_1+\cdots+\tau_n\) for \(n\ge1\), and
\((\tau_i)_{i\ge1}\) are iid random variables satisfying
\(0\le\tau_1<\infty\) almost surely and
\(\mathbb{P}(\tau_1>0)>0\).}
A canonical example is provided by regularly varying interarrival times   with tail index \(\gamma\in(0,1)\).
 In this regime, the renewal count does not concentrate
relatively around its expectation, and such models fall outside the
relative-concentration framework described above.

A recent approach is provided by Cristadoro and Pozzoli~\cite{CristadoroPozzoli2026}, who obtain precise asymptotics through a uniform Tauberian theorem. Their framework already allows random counts that need not concentrate around their expectations and also includes examples in which the random count itself has infinite mean. For finite-mean counts, they formulate sufficient conditions in terms of the probability generating function and, in particular, normalized moments of the count. Their paper does not, however, establish a general  result for arbitrary  renewal counts. In addition, the part of their theory relevant to the regularly varying sums considered there is restricted to tails with  index  \(\alpha\in(0,2)\setminus\{1\}\).

The present paper does not aim to sharpen existing ratio equivalences. Instead, we develop an alternative logarithmic estimate under a scale condition on the random count. A main advantage of our approach is that it applies in several situations not covered by the classical ratio-based framework. 
In particular, our result applies to every  renewal count without restrictions on the interarrival distribution, and shows that the mean count remains an admissible  count-scale even when relative concentration around that mean \eqref{eq:concentration}  fails. Moreover, the framework only requires a tail condition weaker than regular variation, covering every $\alpha>0$, including $\alpha=1$ and $\alpha\ge2$. 
 We also provide an example of counts for which the asymptotic ratio equivalence \eqref{eq:ratioAsymptotics} fails, while our logarithmic estimate remains valid.

More precisely, our count-scale condition  assumes the existence of a deterministic scale
\(\nu(t)\to\infty\) such that, for every \(\varepsilon>0\),
\[
        \Pp\left(
        N(t)<\nu(t)^{1-\varepsilon}
        \right)
        \longrightarrow0,
\]
whereas, for every \(\varepsilon>0\) and \(h>0\),
\[
        \Pp\left(
        N(t)>\nu(t)^{1+\varepsilon}
        \right)
        =
        o\left(\nu(t)^{-h}\right).
\]
The two requirements play different roles. The lower estimate requires,
for every \(\varepsilon>0\), at least
\(\nu(t)^{1-\varepsilon}\) summands to be present with probability
tending to one, whereas the upper estimate excludes count excursions
capable of competing with the one-big-jump mechanism at the logarithmic
scale. Thus the condition controls the logarithmic order of \(N(t)\), rather than
requiring relative concentration around a deterministic first-order
approximation. 

Let the summands \(X_1,X_2,\ldots\) be iid nonnegative random variables, independent of the count, with common distribution function \(F\) and survival function \(\bar F\) satisfying
\[
\bar{F}(x)=x^{-\alpha+o(1)},
        \qquad \alpha>0.
\]
This condition is satisfied whenever the tails are regularly varying with index $\alpha$.  
For \(y>1\), let
\[
        a_y
        =
        F^{\leftarrow}\left(1-\frac1y\right)
\]
denote the corresponding high-quantile scale. Our main result
establishes that, for every count satisfying the preceding count-scale condition,
\[
        \Pp\left(
        S(t)>a_{\nu(t)}\nu(t)^x
        \right)
        =
        \nu(t)^{-\alpha x+o(1)},
        \qquad
        x>\left(1-\frac1\alpha\right)^+.
\]

For  renewal counts, the  count-scale condition holds
automatically. No moment, tail-regularity, or domain-of-attraction
assumption is imposed on the interarrival distribution. 
An explicit admissible scale is
\[
        \nu_N(t)
        =
        \frac{1}{1-\Ee[e^{-\tau_1/t}]}.
\]
This intrinsic scale is defined directly from the interarrival
distribution and does not require an asymptotic equivalent for the
renewal function. Moreover, for the mean renewal count
\(m(t)=\Ee[N(t)]\), we prove the comparison
\[
        (1-e^{-1})\nu_N(t)
        \le
        m(t)+1
        \le
        2\nu_N(t).
\]
As a consequence, the mean renewal count itself is an admissible
count-scale, giving
\[
        \Pp\left(
        S(t)>a_{m(t)}m(t)^x
        \right)
        =
        m(t)^{-\alpha x+o(1)},
        \qquad
        x>\left(1-\frac1\alpha\right)^+.
\]
This remains true even when the relative-concentration
property~\eqref{eq:concentration} fails. The mean count therefore
provides the correct  count-scale without necessarily providing
a relative approximation to the renewal count.
 Thereby, the same framework covers renewal processes with both finite-mean
and infinite-mean interarrival distributions, including models arising
in fractional renewal risk, bursty event data, and high-frequency
finance
\cite{BiardSaussereau2014,KumarLeonenkoPichler2020,Barabasi2005,
MainardiRabertoGorenfloScalas2000,
ScalasGorenfloLuckockMainardiMantelliRaberto2004}.

Related but distinct approaches include logarithmic tail estimates
for independently stopped sums due to Lehtomaa~\cite{Lehtomaa2021}
and large-deviation principles for renewal--reward processes
established by Zamparo~\cite{Zamparo2023}. These works concern
different normalizations, asymptotic scales, or classes of stopping
mechanisms.

As an application, we also develop consequences for finite heterogeneous insurance portfolios. Different blocks may have different polynomial claim-size tail orders, count scales, and exposure-growth rates. Our results provide a simple calculus for combining blockwise logarithmic shortfall costs. Under the stated growth conditions, finite unions of block shortfalls and nonnegative weighted aggregate losses are governed, at the logarithmic level, by the marginal block with the smallest cost. Under independent blocks, simultaneous shortfall costs add, and the probability that at least \(q\) blocks experience a shortfall is governed by the \(q\) smallest marginal costs. These estimates yield explicit blockwise and portfolio-level capital-calibration rules and identify the dominant routes to heterogeneous portfolio shortfalls.

The rest of the paper is organized as follows.
Section~\ref{sec:main-random-sums} introduces the 
count-scale condition, states the main random-sum result, and
illustrates the distinction between logarithmic and precise
asymptotics.
Section~\ref{sec:counting-examples} studies  renewal counts
and establishes the universal validity of the intrinsic and
mean-count scales.
Section~\ref{sec:heterogeneous-systems} develops the heterogeneous
insurance portfolio results. The proofs of the main scalar
large-deviation statements and the auxiliary results are collected
in Appendix~\ref{app:proofs}.

\section{Main result	}
\label{sec:main-random-sums}

Throughout this section, let \(X_1,X_2,\ldots\) be iid nonnegative random
variables with distribution function \(F\) and survival function
\(\bar F=1-F\). We assume that \(\bar F\) has polynomial tail order
\(\alpha>0\), in the sense that
\begin{equation}
\label{eq:polynomial-tail-order}
    \lim_{x\to\infty}\frac{\log\bar F(x)}{\log x}=-\alpha.
\end{equation}
Equivalently, \(\bar F(x)=x^{-\alpha+o(1)}\). Regularly varying tails with index \(\alpha\) satisfy this condition; see
\cite[Chapter~2]{Resnick2007}. The converse need not hold. For instance,
the classical St.~Petersburg distribution\footnote{The classical
St.~Petersburg distribution is defined by
 $\Pp(X=2^k)=2^{-k}$, $k=1,2,\ldots$.}
satisfies \eqref{eq:polynomial-tail-order} with \(\alpha=1\), but its
tail is not regularly varying; see \cite{Berkes2019}.

Let \((N(t))_{t\ge0}\) be a family of nonnegative integer-valued random
variables, independent of \((X_k)_{k\ge1}\), and define the randomly
stopped sum
\[
    S(t):=\sum_{k=1}^{N(t)}X_k.
\]
For \(y>1\), set
\[
    a_y:=F^{\leftarrow}\left(1-\frac1y\right),
    \qquad
    F^{\leftarrow}(u):=\inf\{x\in\R:F(x)\ge u\},
\]
and let \(\ell_\alpha:=(1-1/\alpha)^+\).

We impose the following condition on the count.

\begin{assumption}
\label{ass:count-scale}
There exists a deterministic function \(\nu:[0,\infty)\to(0,\infty)\)
with \(\nu(t)\to\infty\) such that
\[
    \Pp\left(N(t)<\nu(t)^{1-\varepsilon}\right)\longrightarrow0
    \qquad\text{for every }\varepsilon>0,
\]
and
\[
    \Pp\left(N(t)>\nu(t)^{1+\varepsilon}\right)
    =o\left(\nu(t)^{-h}\right)
    \qquad\text{for every }\varepsilon,h>0.
\]
\end{assumption}

We now state the main logarithmic large-deviation result. Its proof is
given in Appendix~\ref{app:proofs}.

\begin{theorem}
\label{thm:random-sum-main}
Suppose that \(N(t)\) satisfies Assumption~\ref{ass:count-scale}.
Then, for every \(x>\ell_\alpha\),
\[
    \lim_{t\to\infty}\frac{1}{\log\nu(t)}
    \log\Pp\left(S(t)>a_{\nu(t)}\nu(t)^x\right)=-\alpha x.
\]
Equivalently,
\begin{equation}
\label{eq:MainEstimate}
    \Pp\left(S(t)>a_{\nu(t)}\nu(t)^x\right)
    =\nu(t)^{-\alpha x+o(1)}.
\end{equation}
If, in addition, \(\log\nu(t)/\log t\to\rho\in(0,\infty)\), then
\[
    \Pp\left(S(t)>a_{\nu(t)}\nu(t)^x\right)
    =t^{-\rho\alpha x+o(1)}
    \qquad\text{for every }x>\ell_\alpha.
\]
\end{theorem}

\begin{remark}
\label{rem:general-thresholds}
The threshold in Theorem~\ref{thm:random-sum-main} can be replaced by
any threshold of the same logarithmic order relative to
\(a_{\nu(t)}\). More precisely, let \(u_t>0\) satisfy
\[
    \lim_{t\to\infty}
    \frac{\log\bigl(u_t/a_{\nu(t)}\bigr)}{\log\nu(t)}=x
\]
for some \(x>\ell_\alpha\). Then,\footnote{Indeed, fix
\(\varepsilon>0\) such that \(x-\varepsilon>\ell_\alpha\).
For all sufficiently large \(t\),
\[
    a_{\nu(t)}\nu(t)^{x-\varepsilon}
    \le u_t\le a_{\nu(t)}\nu(t)^{x+\varepsilon}.
\]
Consequently,
\[
    \Pp\left(S(t)>a_{\nu(t)}\nu(t)^{x+\varepsilon}\right)
    \le\Pp\left(S(t)>u_t\right)
    \le\Pp\left(S(t)>a_{\nu(t)}\nu(t)^{x-\varepsilon}\right).
\]
Applying Theorem~\ref{thm:random-sum-main} and letting
\(\varepsilon\downarrow0\) gives the assertion.}
\[
    \Pp\left(S(t)>u_t\right)=\nu(t)^{-\alpha x+o(1)}.
\]
\end{remark}

\begin{examples}
\label{examples}
\begin{enumerate}
\item \emph{Deterministic counts.}
Let \(N(t)\) be a deterministic integer-valued function with
\(N(t)\to\infty\). Then Assumption~\ref{ass:count-scale} holds with
\(\nu(t)=N(t)\) for sufficiently large \(t\).

Indeed, fix \(\varepsilon>0\). Since \(\nu(t)\to\infty\), we have
\[
    \nu(t)^{1-\varepsilon}<\nu(t)<\nu(t)^{1+\varepsilon}
\]
for all sufficiently large \(t\). Hence, using \(N(t)=\nu(t)\),
\[
    \Pp\left(N(t)<\nu(t)^{1-\varepsilon}\right)
    =\Pp\left(N(t)>\nu(t)^{1+\varepsilon}\right)=0
\]
for such \(t\), which proves both conditions in
Assumption~\ref{ass:count-scale}.

\item \emph{Poisson counts.}
Let \(N(t)\) be a non-homogeneous Poisson process with deterministic
cumulative intensity \(\Lambda(t)\to\infty\). Then
Assumption~\ref{ass:count-scale} holds with \(\nu(t)=\Lambda(t)\)
for sufficiently large \(t\).

To that end, write \(\nu(t)=\Lambda(t)\), so that
\[
    N(t)\sim\operatorname{Poisson}(\nu(t)).
\]
Recall the Chernoff bound for Poisson distributions
\cite[Theorem~5.4]{MitzenmacherUpfal2005}. If
\(X\sim\operatorname{Poisson}(\mu)\), then
\[
    \Pp(X\ge x)
    \le\exp\left\{-\mu H\left(\frac{x}{\mu}\right)\right\},
    \qquad x>\mu,
\]
and the corresponding lower-tail bound holds for \(0<x<\mu\), with
\(\Pp(X\le x)\), where \(H(y):=y\log y-y+1\).

Fix \(\varepsilon>0\). For the lower-count estimate, set
\(x_t:=\nu(t)^{1-\varepsilon}\). Since
\[
    \frac{x_t}{\nu(t)}=\nu(t)^{-\varepsilon}\longrightarrow0
\]
and \(H(y)\to1\) as \(y\downarrow0\), we obtain
\[
\begin{aligned}
    \Pp\left(N(t)<\nu(t)^{1-\varepsilon}\right)
    &\le\exp\left\{-\nu(t)H\bigl(\nu(t)^{-\varepsilon}\bigr)\right\} \\
    &\le\exp\left\{-\frac12\nu(t)\right\}\longrightarrow0
\end{aligned}
\]
for all sufficiently large \(t\).

As for the upper-count estimate, set
\(y_t:=\nu(t)^{1+\varepsilon}\). Since
\[
    H\bigl(\nu(t)^\varepsilon\bigr)
    \ge\frac{\varepsilon}{2}\nu(t)^\varepsilon\log\nu(t)
\]
for all sufficiently large \(t\), the Chernoff bound gives
\[
\begin{aligned}
    \Pp\left(N(t)>\nu(t)^{1+\varepsilon}\right)
    &\le\exp\left\{-\nu(t)H\bigl(\nu(t)^\varepsilon\bigr)\right\} \\
    &\le\exp\left\{
        -\frac{\varepsilon}{2}\nu(t)^{1+\varepsilon}\log\nu(t)
        \right\} \\
        &=\nu(t)^{-\tfrac{\varepsilon}{2}\nu(t)^{1+\varepsilon}}\\
    &=o\left(\nu(t)^{-h}\right)
\end{aligned}
\]
for every \(h>0\). This shows that
Assumption~\ref{ass:count-scale} holds with \(\nu(t)=\Lambda(t)\).
In particular, for a homogeneous Poisson process with rate
\(\lambda>0\), we can take \(\nu(t)=\lambda t\) for \(t>0\).
\end{enumerate}
\end{examples}

\begin{remark}
For comparison, recall Cram\'er's theorem for light-tailed sums
\cite{Cramer1938,DemboZeitouni1998}. Suppose that
\(\Lambda_X(y):=\log\Ee[e^{yX_1}]\) is finite in a neighbourhood of
the origin, and let \(I(x):=\sup_{y\in\R}\{xy-\Lambda_X(y)\}\).
Then,
\[
    \lim_{n\to\infty}\frac1n
    \log\Pp\left(\sum_{i=1}^nX_i\ge nx\right)=-I(x),
    \qquad x>\Ee[X_1].
\]
For upper tails satisfying \eqref{eq:polynomial-tail-order},
\(\Ee[e^{yX_1}]=\infty\) for every \(y>0\), so that the Cram\'er
scale is no longer appropriate. In this
case, taking \(N(t)=\lfloor t\rfloor\) in
Theorem~\ref{thm:random-sum-main}, as in Examples~\ref{examples},
gives instead
\[
    \lim_{n\to\infty}\frac1{\log n}
    \log\Pp\left(\sum_{i=1}^nX_i\ge a_nn^x\right)=-\alpha x,
    \qquad x>\ell_\alpha.
\]
Thus, Theorem~\ref{thm:random-sum-main} provides a logarithmic
counterpart of the upper-tail estimate in Cram\'er's theorem: the speed
\(n\) and the rate function \(I\) are replaced by the speed \(\log n\)
and the linear exponent \(x\mapsto\alpha x\) on \(x>\ell_\alpha\),
respectively.
\end{remark}

In the finite-mean setting, the classical approach to heavy-tailed
randomly stopped sums seeks conditions under which an  asymptotic equivalence 
of the form
\[
    \Pp\left(S(t)-\mu m(t)>u\right)\sim m(t)\bar F(u)
\]
holds, where \(\mu:=\Ee[X_1]<\infty\) and
\(m(t):=\Ee[N(t)]\); see, e.g., \cite{KluppelbergMikosch1997}.
The following example concerns the uncentered mean-based equivalence.
It shows that rare but very large values of the count can change the
precise constant without changing the logarithmic tail exponent.
Thus, the mean count may remain an admissible count-scale even
when the asymptotic equivalence fails.

\begin{example}
\label{ex:logarithmic-not-precise}
For \(n\ge2\), let \(N_n\), independent of the summands, satisfy
\[
    \Pp(N_n=n)=1-e^{-n},
    \qquad
    \Pp(N_n=L_n)=e^{-n},
    \qquad
    L_n:=\lceil ne^n\rceil.
\]
Then,
\[
    m(n):=\Ee[N_n]
    =n(1-e^{-n})+e^{-n}L_n\sim2n.
\]
Define \(N(t):=N_{\lfloor t\rfloor}\) for \(t\ge2\). We first show
that \(N(t)\) satisfies Assumption~\ref{ass:count-scale} with the
mean-count scale \(m(t):=\Ee[N(t)]\).

Fix \(\varepsilon>0\). Since \(m(n)\sim2n\), we have
\[
    \frac{m(n)^{1-\varepsilon}}n
    \sim2^{1-\varepsilon}n^{-\varepsilon}\longrightarrow0.
\]
Consequently, \(m(n)^{1-\varepsilon}<n\) for all sufficiently large
\(n\), and therefore
\[
    \Pp\left(N_n<m(n)^{1-\varepsilon}\right)=0.
\]
On the other hand,
\[
    \frac{m(n)^{1+\varepsilon}}n
    \sim2^{1+\varepsilon}n^\varepsilon\longrightarrow\infty,
    \qquad
    \frac{m(n)^{1+\varepsilon}}{L_n}
    \sim2^{1+\varepsilon}n^\varepsilon e^{-n}\longrightarrow0.
\]
Hence, for all sufficiently large \(n\),
\[
    n<m(n)^{1+\varepsilon}<L_n,
\]
and we obtain
\[
    \Pp\left(N_n>m(n)^{1+\varepsilon}\right)
    =e^{-n}=o\left(m(n)^{-h}\right)
    \qquad\text{for every }h>0.
\]
This proves Assumption~\ref{ass:count-scale} with the mean-count scale.

Now suppose that the summands have a Pareto Type I distribution on
\([1,\infty)\), with survival function
\[
    \bar F(u)=u^{-1/2},\qquad u\ge1,
\]
and set \(u_n:=n^4\). Since
\[
    n\bar F(n^4)=n^{-1}\longrightarrow0,
\]
the estimate for deterministic sums in
\cite[Theorem~9.1]{DenisovDiekerShneer2008} gives
\[
    \Pp(S_n>n^4)\sim n\bar F(n^4)=n^{-1}.
\]
Moreover, \(L_n>n^4\) for all sufficiently large \(n\). Since
\(X_k\ge1\) a.s., it follows that
\[
    \Pp(S_{L_n}>n^4)=1
\]
for such \(n\). By independence of the count and the summands,
\[
\begin{aligned}
    \Pp(S_{N_n}>n^4)
    &=(1-e^{-n})\Pp(S_n>n^4)
        +e^{-n}\Pp(S_{L_n}>n^4) \\
    &=(1-e^{-n})\Pp(S_n>n^4)+e^{-n} \\
    &\sim n^{-1}.
\end{aligned}
\]
On the other hand,
\[
    m(n)\bar F(n^4)\sim2n\cdot n^{-2}=\frac2n.
\]
Consequently,
\[
    \frac{\Pp(S_{N_n}>n^4)}{m(n)\bar F(n^4)}
    \longrightarrow\frac12.
\]
This shows that Assumption~\ref{ass:count-scale} does not imply the 
asymptotic equivalence
\[
    \Pp(S_{N_n}>u_n)\sim m(n)\bar F(u_n).
\]

In contrast, for this Pareto distribution, \(a_y=y^2\), and therefore
\[
    \frac{\log(n^4/a_{m(n)})}{\log m(n)}
    =\frac{\log(n^4/m(n)^2)}{\log m(n)}\longrightarrow2.
\]
Since \(\alpha=1/2\), it follows from
Remark~\ref{rem:general-thresholds} that
\[
    \Pp(S_{N_n}>n^4)=m(n)^{-1+o(1)},
\]
in agreement with the preceding calculation.

Notice also that the count does not concentrate relatively around its
mean. Indeed, \(\Pp(N_n=n)=1-e^{-n}\to1\) and \(m(n)\sim2n\) imply
\[
    \frac{N_n}{m(n)}\xrightarrow{\mathrm P}\frac12\ne1.
\]
\end{example}

\section{Consequences for renewal counts}
\label{sec:counting-examples}

In this section, we show that the mean count is an admissible count-scale
for every  renewal count. Throughout, let \((\tau_i)_{i\ge1}\)
be iid interarrival times such that \(0\le\tau_1<\infty\) a.s. and
\(\Pp(\tau_1>0)>0\). We define
\[
    T_0=0,
    \qquad
    T_n=\tau_1+\cdots+\tau_n,
    \quad n\ge1,
\]
and
\[
    N(t)=\max\{n\ge0:T_n\le t\},
    \qquad t\ge0.
\]
Here, \(T_n\) is the time of the \(n\)-th renewal and \(N(t)\) is the
number of renewals up to time \(t\). Simultaneous renewals are counted
separately, including any renewals at time zero.

\begin{remark}
Since \(\Pp(\tau_1>0)>0\), there exists \(a>0\) such that
\(\Pp(\tau_1\ge a)>0\). By independence and the Borel--Cantelli lemma,
\(\tau_i\ge a\) occurs infinitely often on a set of probability 1.\footnote{Indeed, the events
\(A_i=\{\tau_i\ge a\}\) are independent and satisfy
\(\Pp(A_i)=\Pp(\tau_1\ge a)=p>0\). Hence
\(\sum_{i=1}^{\infty}\Pp(A_i)=\infty\), and the second Borel--Cantelli
lemma gives
\(\Pp(\bigcap_{j\ge1}\bigcup_{i\ge j}A_i)=1\).}
Since the interarrival times are nonnegative, it follows that
\[
   \lim_{n\to\infty} T_n=\infty
    \qquad\text{a.s.}
\]
In particular, \(N(t)<\infty\) a.s. for every \(t\ge0\).
\end{remark}

We denote the mean count by \(m(t):=\Ee[N(t)]\), and define the
intrinsic renewal scale by
\[
    \nu_N(t):=\frac{1}{1-\Ee[e^{-\tau_1/t}]},
    \qquad t>0.
\]
The summands \(X_1,X_2,\ldots\) and their independence from the renewal
count are as in Section~\ref{sec:main-random-sums}.

We first show that \(\nu_N\) is an admissible count-scale.

\begin{proposition}
\label{prop:renewal-laplace-scale}
The renewal count \(N(t)\) satisfies Assumption~\ref{ass:count-scale}
with the scale \(\nu_N(t)\). In particular, under the standing assumptions,
\[
    \Pp\left(S(t)>a_{\nu_N(t)}\nu_N(t)^x\right)
    =\nu_N(t)^{-\alpha x+o(1)},
    \qquad x>\ell_\alpha.
\]
\end{proposition}

\begin{proof}
Let \(\varphi(s):=\Ee[e^{-s\tau_1}]\) for \(s>0\). Since
\(\tau_1<\infty\) a.s., dominated convergence gives
\(\varphi(1/t)\to1\) as \(t\to\infty\). Moreover,
\(\Pp(\tau_1>0)>0\) implies \(\varphi(1/t)<1\) for every \(t>0\).
Consequently,
\[
    \nu_N(t)=\frac{1}{1-\varphi(1/t)}\longrightarrow\infty.
\]

Fix \(\varepsilon>0\). We first prove the lower-count estimate. Set
\[
    n_-(t):=\left\lceil\nu_N(t)^{1-\varepsilon}\right\rceil.
\]
By definition of \(N(t)\), the event
\(\{N(t)<\nu_N(t)^{1-\varepsilon}\}\) is contained in
\(\{T_{n_-(t)}>t\}\). Using the inequality
\[
    \one_{\{z>1\}}\le\frac{1-e^{-z}}{1-e^{-1}},
    \qquad z\ge0,
\]
and independence of the interarrival times, we obtain
\[
\begin{aligned}
    \Pp\left(N(t)<\nu_N(t)^{1-\varepsilon}\right)
    &\le \Pp(T_{n_-(t)}>t) \\
    &\le \frac{1}{1-e^{-1}}
        \Ee\left[1-e^{-T_{n_-(t)}/t}\right] \\
    &=\frac{1-\varphi(1/t)^{n_-(t)}}{1-e^{-1}} \\
    &\le\frac{n_-(t)}{1-e^{-1}}\bigl(1-\varphi(1/t)\bigr) \\
    &=\frac{1}{1-e^{-1}}\frac{n_-(t)}{\nu_N(t)} \\
    &\le\frac{1}{1-e^{-1}}
        \left(\nu_N(t)^{-\varepsilon}+\nu_N(t)^{-1}\right)
        \longrightarrow0,
\end{aligned}
\]
where the third inequality follows from
\(1-u^n\le n(1-u)\) for \(u\in[0,1]\) and \(n\ge1\).

As for the upper-count estimate, set
\[
    n_+(t):=\left\lfloor\nu_N(t)^{1+\varepsilon}\right\rfloor+1.
\]
Then \(\{N(t)>\nu_N(t)^{1+\varepsilon}\}\) is contained in
\(\{T_{n_+(t)}\le t\}\). By Markov's inequality and independence,
\[
\begin{aligned}
    \Pp\left(N(t)>\nu_N(t)^{1+\varepsilon}\right)
    &\le\Pp(T_{n_+(t)}\le t) \\
    &=\Pp\left(e^{-T_{n_+(t)}/t}\ge e^{-1}\right) \\
    &\le e\,\Ee[e^{-T_{n_+(t)}/t}] \\
    &=e\,\varphi(1/t)^{n_+(t)}.
\end{aligned}
\]
Since \(\log u\le-(1-u)\) for \(0<u\le1\), it follows that
\[
\begin{aligned}
    \Pp\left(N(t)>\nu_N(t)^{1+\varepsilon}\right)
    &\le\exp\left\{1-n_+(t)\bigl(1-\varphi(1/t)\bigr)\right\} \\
    &=\exp\left\{1-\frac{n_+(t)}{\nu_N(t)}\right\} \\
    &\le\exp\left\{1-\nu_N(t)^\varepsilon\right\}.
\end{aligned}
\]
For every \(h>0\), we have
\[
    \nu_N(t)^h\exp\left\{1-\nu_N(t)^\varepsilon\right\}
    \longrightarrow0,
\]
and therefore
\[
    \Pp\left(N(t)>\nu_N(t)^{1+\varepsilon}\right)
    =o\left(\nu_N(t)^{-h}\right).
\]
This shows that Assumption~\ref{ass:count-scale} holds with the scale
\(\nu_N(t)\). The assertion concerning \(S(t)\) now follows from
Theorem~\ref{thm:random-sum-main}.
\end{proof}

The following lemma compares the intrinsic renewal scale with the mean
count.

\begin{lemma}
\label{lem:renewal-scales}
For every \(t>0\),
\[
    (1-e^{-1})\nu_N(t)\le m(t)+1\le2\nu_N(t).
\]
\end{lemma}

\begin{proof}
Fix \(t>0\), and set
\[
    A(t):=\Ee[\tau_1\wedge t],
    \qquad
    K(t):=N(t)+1=\min\{n\ge1:T_n>t\}.
\]
Notice that \(A(t)>0\) since \(\Pp(\tau_1>0)>0\). For \(r\in\mathbb N\),
let \(K_r(t):=K(t)\wedge r\). Then
\[
    \sum_{i=1}^{K_r(t)}(\tau_i\wedge t)
    =\sum_{i=1}^r(\tau_i\wedge t)\one_{\{K(t)\ge i\}}.
\]
Since \(\{K(t)\ge i\}=\{T_{i-1}\le t\}\) is independent of \(\tau_i\),
we obtain
\[
\begin{aligned}
    \Ee\left[\sum_{i=1}^{K_r(t)}(\tau_i\wedge t)\right]
    &=\sum_{i=1}^r\Ee[\tau_i\wedge t]\,\Pp(K(t)\ge i) \\
    &=A(t)\sum_{i=1}^r\Pp(K(t)\ge i) \\
    &=A(t)\Ee[K_r(t)].
\end{aligned}
\]
Moreover, it holds pathwise that
\[
    \sum_{i=1}^{K_r(t)}(\tau_i\wedge t)\le2t.
\]
Indeed, on \(\{K(t)>r\}\) we have \(T_r\le t\). On
\(\{K(t)\le r\}\), the first \(K(t)-1\) interarrival times have sum
at most \(t\), and the last truncated interarrival time is at most
\(t\). Consequently,
\[
    A(t)\Ee[K_r(t)]\le2t.
\]
Letting \(r\to\infty\), monotone convergence gives
\[
    \Ee[K(t)]\le\frac{2t}{A(t)}<\infty.
\]

We next pass to the limit in the finite-sum identity above. By dominated
convergence on the left-hand side and monotone convergence on the
right-hand side,
\[
    \Ee\left[\sum_{i=1}^{K(t)}(\tau_i\wedge t)\right]
    =A(t)\Ee[K(t)].
\]
Since \(T_{K(t)-1}\le t<T_{K(t)}\), we have
\[
    t\le\sum_{i=1}^{K(t)}(\tau_i\wedge t)\le2t.
\]
It follows that
\[
    \frac{t}{A(t)}\le\Ee[K(t)]=m(t)+1\le\frac{2t}{A(t)}.
\]

On the other hand, the inequalities
\[
    (1-e^{-1})(z\wedge1)\le1-e^{-z}\le z\wedge1,
    \qquad z\ge0,
\]
applied to \(z=\tau_1/t\), yield
\[
    (1-e^{-1})\frac{A(t)}t
    \le1-\Ee[e^{-\tau_1/t}]
    \le\frac{A(t)}t.
\]
Hence
\[
    \frac{t}{A(t)}\le\nu_N(t)
    \le\frac{1}{1-e^{-1}}\frac{t}{A(t)}.
\]
Combining the two comparisons, we conclude that
\[
    (1-e^{-1})\nu_N(t)\le m(t)+1\le2\nu_N(t).
\]
\end{proof}

We can now show that the mean count is an admissible deterministic
scale.

\begin{theorem}
\label{thm:renewal-mean-scale}
For every \(t>0\), \(m(t)<\infty\), and \(m(t)\to\infty\) as
\(t\to\infty\). Moreover, \(N(t)\) satisfies
Assumption~\ref{ass:count-scale} with the mean-count scale
\(\nu(t)=m(t)\) for sufficiently large \(t\). In particular, under the
standing assumptions,
\[
    \Pp\left(S(t)>a_{m(t)}m(t)^x\right)
    =m(t)^{-\alpha x+o(1)},
    \qquad x>\ell_\alpha.
\]
\end{theorem}

\begin{proof}
By Lemma~\ref{lem:renewal-scales},
\[
    (1-e^{-1})\nu_N(t)\le m(t)+1\le2\nu_N(t).
\]
Since \(\nu_N(t)<\infty\) for every \(t>0\), it follows that
\(m(t)<\infty\). Moreover, \(\nu_N(t)\to\infty\) by
Proposition~\ref{prop:renewal-laplace-scale}, so that the lower bound
gives \(m(t)\to\infty\). With \(c_0:=(1-e^{-1})/2\), we therefore have
\[
    c_0\nu_N(t)\le m(t)\le2\nu_N(t)
\]
for all sufficiently large \(t\).

Fix \(0<\varepsilon<1\). Since \(\nu_N(t)\to\infty\), for all
sufficiently large \(t\),
\[
    m(t)^{1-\varepsilon}
    \le\bigl(2\nu_N(t)\bigr)^{1-\varepsilon}
    \le\nu_N(t)^{1-\varepsilon/2}.
\]
It follows from Proposition~\ref{prop:renewal-laplace-scale} that
\[
\begin{aligned}
    \Pp\left(N(t)<m(t)^{1-\varepsilon}\right)
    &\le\Pp\left(N(t)<\nu_N(t)^{1-\varepsilon/2}\right)
    \longrightarrow0.
\end{aligned}
\]
Similarly, for all sufficiently large \(t\),
\[
    m(t)^{1+\varepsilon}
    \ge\bigl(c_0\nu_N(t)\bigr)^{1+\varepsilon}
    \ge\nu_N(t)^{1+\varepsilon/2}.
\]
Thus, for every \(h>0\),
\[
\begin{aligned}
    &m(t)^h\Pp\left(N(t)>m(t)^{1+\varepsilon}\right) \\
    &\qquad\le2^h\nu_N(t)^h
        \Pp\left(N(t)>\nu_N(t)^{1+\varepsilon/2}\right)
        \longrightarrow0,
\end{aligned}
\]
again by Proposition~\ref{prop:renewal-laplace-scale}. This proves
\[
    \Pp\left(N(t)>m(t)^{1+\varepsilon}\right)
    =o\left(m(t)^{-h}\right)
\]
for every \(h>0\).

For \(\varepsilon\ge1\), the estimates follow from the case
\(\varepsilon=1/2\). Indeed, since \(m(t)\to\infty\), the inclusions
\[
    \{N(t)<m(t)^{1-\varepsilon}\}
    \subseteq\{N(t)<m(t)^{1/2}\}
\]
and
\[
    \{N(t)>m(t)^{1+\varepsilon}\}
    \subseteq\{N(t)>m(t)^{3/2}\}
\]
hold for all sufficiently large \(t\). Hence
Assumption~\ref{ass:count-scale} holds with the mean-count scale.
The assertion concerning \(S(t)\) follows from
Theorem~\ref{thm:random-sum-main}.
\end{proof}

We next consider interarrival distributions with regularly varying tails.

\begin{corollary}
\label{cor:rv-counts}
Suppose that the survival function \(\bar F_\tau\) of \(\tau_1\) is
regularly varying with tail index \(\gamma\in(0,1)\). Then
\[
    \nu_N(t)\sim\frac{1}{\Gamma(1-\gamma)\bar F_\tau(t)}.
\]
In particular, \(\nu_N\) is regularly varying with index \(\gamma\), and
\[
    \frac{\log\nu_N(t)}{\log t}\longrightarrow\gamma,
    \qquad
    \frac{\log m(t)}{\log t}\longrightarrow\gamma.
\]
Consequently, under the standing assumptions on the summands,
\[
    \Pp\left(S(t)>a_{\nu_N(t)}\nu_N(t)^x\right)
    =t^{-\gamma\alpha x+o(1)},
    \qquad x>\ell_\alpha,
\]
and
\[
    \Pp\left(S(t)>a_{m(t)}m(t)^x\right)
    =t^{-\gamma\alpha x+o(1)},
    \qquad x>\ell_\alpha.
\]
\end{corollary}

\begin{proof}
By Karamata's Tauberian theorem
\cite[Corollary~8.1.7]{BinghamGoldieTeugels1989},
\[
    1-\Ee[e^{-\tau_1/t}]
    \sim\Gamma(1-\gamma)\bar F_\tau(t).
\]
Hence
\[
    \nu_N(t)=\frac{1}{1-\Ee[e^{-\tau_1/t}]}
    \sim\frac{1}{\Gamma(1-\gamma)\bar F_\tau(t)}.
\]
Since \(\bar F_\tau\) has tail index \(\gamma\), it follows that
\(\nu_N\) is regularly varying with index \(\gamma\), and therefore
\[
    \frac{\log\nu_N(t)}{\log t}\longrightarrow\gamma.
\]

Moreover, Lemma~\ref{lem:renewal-scales} gives
\[
    (1-e^{-1})\nu_N(t)\le m(t)+1\le2\nu_N(t).
\]
Since \(m(t)\to\infty\) and \(\nu_N(t)\to\infty\), we obtain
\[
    \frac{\log m(t)}{\log\nu_N(t)}\longrightarrow1.
\]
Consequently,
\[
    \frac{\log m(t)}{\log t}
    =\frac{\log m(t)}{\log\nu_N(t)}
        \frac{\log\nu_N(t)}{\log t}
    \longrightarrow\gamma.
\]

Finally, Proposition~\ref{prop:renewal-laplace-scale} and
Theorem~\ref{thm:renewal-mean-scale} imply, respectively,
\[
    \Pp\left(S(t)>a_{\nu_N(t)}\nu_N(t)^x\right)
    =\nu_N(t)^{-\alpha x+o(1)},
    \qquad x>\ell_\alpha,
\]
and
\[
    \Pp\left(S(t)>a_{m(t)}m(t)^x\right)
    =m(t)^{-\alpha x+o(1)},
    \qquad x>\ell_\alpha.
\]
The proof is complete.
\end{proof}

The following example shows that the mean-count scale may be admissible
even when the renewal count does not concentrate relatively around its
mean. In particular, the relative-concentration assumption used in
\cite{KluppelbergMikosch1997} need not hold.

\begin{example}
\label{ex:pareto-renewal-not-KM}
Suppose that the interarrival times have a Pareto Type I distribution
on \([1,\infty)\), with survival function
\[
    \Pp(\tau_1>u)=u^{-\gamma},
    \qquad u\ge1,
    \quad 0<\gamma<1.
\]
Then \(\Ee[\tau_1]=\infty\), and Corollary~\ref{cor:rv-counts} gives
\[
    \frac{\log m(t)}{\log t}\longrightarrow\gamma.
\]
Combining this with Theorem~\ref{thm:renewal-mean-scale}, we obtain
\[
    \Pp\left(S(t)>a_{m(t)}m(t)^x\right)
    =t^{-\gamma\alpha x+o(1)},
    \qquad x>\ell_\alpha.
\]
On the other hand, the interarrival distribution belongs to the domain
of attraction of a positive stable law with index \(\gamma\in(0,1)\).
The corresponding renewal count does not concentrate relatively around
its mean; see \cite[Example~3]{CristadoroPozzoli2026}. That is,
\[
    \frac{N(t)}{m(t)}\not\longrightarrow1
    \qquad\text{in probability}.
\]
Thus Assumption~\ref{ass:count-scale} holds with the mean-count scale
\(m(t)\), although this renewal count lies outside the
relative-concentration framework referred to above.
\end{example}

\section{Heterogeneous insurance portfolios}
\label{sec:heterogeneous-systems}

In this section, we apply the scalar tail estimates to finite insurance
portfolios with different claim-size tails and exposure growth rates.
We obtain capital thresholds for the individual blocks and the aggregate
portfolio, together with shortfall exponents under different assumptions
on dependence.

Let \(m\ge1\) be the fixed number of blocks, which may represent lines
of business or geographical regions. For \(j=1,\ldots,m\), let
\[
    S_j(t):=\sum_{k=1}^{N_j(t)}X_{j,k}
\]
denote the cumulative claims in block \(j\) up to time \(t\). The claim
sizes \((X_{j,k})_{k\ge1}\) are iid and nonnegative, with distribution
function \(F_j\) and survival function \(\bar F_j\) satisfying
\[
    \lim_{x\to\infty}\frac{\log\bar F_j(x)}{\log x}=-\alpha_j,
    \qquad \alpha_j>0.
\]
Within each block, \(N_j(t)\) is independent of the claim sequence and satisfies
Assumption~\ref{ass:count-scale} with a scale \(\nu_j(t)\). Dependence
across blocks is unrestricted unless stated otherwise.

To compare the tail probabilities on a common calendar-time scale,
we assume that
\[
    \frac{\log\nu_j(t)}{\log t}\longrightarrow\rho_j\in(0,\infty),
    \qquad j=1,\ldots,m.
\]
For  renewal counts, Theorem~\ref{thm:renewal-mean-scale}
allows us to take \(\nu_j(t)=\Ee[N_j(t)]\) for sufficiently large
\(t\). Finite-mean interarrival times give \(\rho_j=1\). If the
interarrival survival function is regularly varying with tail index
\(\gamma_j\in(0,1)\), then \(\rho_j=\gamma_j\) by
Corollary~\ref{cor:rv-counts}.

Set
\[
    a_{j,y}:=F_j^{\leftarrow}\left(1-\frac1y\right),
    \qquad
    \ell_j:=\left(1-\frac1{\alpha_j}\right)^+.
\]
It follows from Theorem~\ref{thm:random-sum-main} that
\begin{equation}
\label{eq:blockEst}
    \Pp\left(S_j(t)>a_{j,\nu_j(t)}\nu_j(t)^x\right)
    =t^{-\rho_j\alpha_jx+o(1)},
    \qquad x>\ell_j.
\end{equation}
Throughout this section, capital is a deterministic threshold for the
cumulative claims at horizon \(t\), and a shortfall occurs when claims
exceed this threshold. The calibration rules below specify the
asymptotic decay exponent of the shortfall probability.

\subsection{Block capitals and shortfall exponents}

For each \(j=1,\ldots,m\), fix \(\beta_j>\ell_j\), and define
\[
    K_j(t):=a_{j,\nu_j(t)}\nu_j(t)^{\beta_j},
    \qquad
    A_j(t):=\{S_j(t)>K_j(t)\},
    \qquad
    c_j:=\rho_j\alpha_j\beta_j.
\]
By \eqref{eq:blockEst},
\[
    \Pp(A_j(t))=t^{-c_j+o(1)}.
\]
We refer to \(c_j\) as the logarithmic shortfall cost of block \(j\).
A larger value of \(c_j\) corresponds to faster decay of the shortfall
probability.

For a prescribed target cost \(c_j\), the choice
\(\beta_j=c_j/(\rho_j\alpha_j)\) gives the capital rule
\begin{equation}
\label{eq:block-capital-calibration}
    K_j(t)=a_{j,\nu_j(t)}\nu_j(t)^{c_j/(\rho_j\alpha_j)},
    \qquad c_j>\rho_j(\alpha_j-1)^+.
\end{equation}
The strict inequality is equivalent to \(\beta_j>\ell_j\). Moreover,
Lemma~\ref{lem:quantile-tail} gives
\[
    \frac{\log K_j(t)}{\log t}
    \longrightarrow\frac{\rho_j+c_j}{\alpha_j}.
\]
Thus, for a fixed polynomial tail order and target cost, greater exposure growth
requires a larger capital growth exponent. If \(\alpha_j>1\), the
admissible capital growth exponent is strictly larger than \(\rho_j\).
In particular, when \(\nu_j(t)=\Ee[N_j(t)]\) is the mean renewal count,
the capital threshold has a higher polynomial order than the mean
aggregate claims of the block.


\subsection{Individual and simultaneous block shortfalls}

For every nonempty \(J\subset\{1,\ldots,m\}\), set
\[
    A_J^{\cup}(t):=\bigcup_{j\in J}A_j(t),
    \qquad
    A_J^{\cap}(t):=\bigcap_{j\in J}A_j(t).
\]
These are, respectively, the events that at least one block in \(J\)
experiences a shortfall and that all blocks in \(J\) experience a
shortfall. For \(q=1,\ldots,m\), let
\[
    A^{(q)}(t):=
    \left\{\sum_{j=1}^m\one_{A_j(t)}\ge q\right\}
\]
denote the event that at least \(q\) blocks experience a shortfall.

The first result does not require independence across blocks.

\begin{proposition}[At least one block shortfall]
\label{prop:heterogeneous-robust}
For every nonempty \(J\subset\{1,\ldots,m\}\), under arbitrary
dependence across blocks,
\[
    \Pp(A_J^{\cup}(t))=t^{-\min_{j\in J}c_j+o(1)}.
\]
\end{proposition}

\begin{proof}
Fix a nonempty \(J\subset\{1,\ldots,m\}\), and set
\(c_*:=\min_{j\in J}c_j\). Since
\[
    \max_{j\in J}\Pp(A_j(t))
    \le\Pp(A_J^{\cup}(t))
    \le\sum_{j\in J}\Pp(A_j(t)),
\]
the assertion follows from \(\Pp(A_j(t))=t^{-c_j+o(1)}\) and the
finiteness of \(J\). Indeed, both outer terms are of the form
\(t^{-c_*+o(1)}\).
\end{proof}

For simultaneous shortfalls, we impose independence across blocks.
To that end, define
\[
    \mathcal H_j:=
    \sigma\bigl(N_j(s),X_{j,k}:s\ge0,\ k\ge1\bigr),
    \qquad j=1,\ldots,m.
\]

\begin{proposition}[Simultaneous block shortfalls]
\label{prop:joint-shortfalls}
Suppose that \(\mathcal H_1,\ldots,\mathcal H_m\) are mutually
independent. Then, for every nonempty \(J\subset\{1,\ldots,m\}\),
\[
    \Pp(A_J^{\cap}(t))=t^{-\sum_{j\in J}c_j+o(1)}.
\]
Moreover, if \(c_{(1)}\le\cdots\le c_{(m)}\) are the ordered block
costs, then
\[
    \Pp(A^{(q)}(t))=t^{-\sum_{r=1}^q c_{(r)}+o(1)},
    \qquad q=1,\ldots,m.
\]
\end{proposition}

\begin{proof}
For every nonempty \(J\subset\{1,\ldots,m\}\), independence gives
\[
    \Pp(A_J^{\cap}(t))
    =\prod_{j\in J}\Pp(A_j(t))
    =t^{-\sum_{j\in J}c_j+o(1)}.
\]
As for the second assertion, fix \(q\in\{1,\ldots,m\}\). By definition,
\[
    A^{(q)}(t)=
    \bigcup_{\substack{J\subset\{1,\ldots,m\}\\|J|=q}}
        A_J^{\cap}(t).
\]
Since this union is finite, the same argument as in the proof of
Proposition~\ref{prop:heterogeneous-robust}, applied to the preceding
intersection estimates, gives
\[
    \Pp(A^{(q)}(t))=t^{-c_*+o(1)},
    \qquad
    c_*:=\min_{\substack{J\subset\{1,\ldots,m\}\\|J|=q}}
        \sum_{j\in J}c_j.
\]
The minimum is attained by choosing the \(q\) blocks with the smallest
costs. Hence \(c_*=\sum_{r=1}^q c_{(r)}\), which proves the assertion.
\end{proof}

The preceding results distinguish individual block protection from
protection against simultaneous shortfalls. For instance, suppose that
\[
    c>\max_{1\le j\le m}\rho_j(\alpha_j-1)^+
\]
and that every block is calibrated to cost \(c\) by
\eqref{eq:block-capital-calibration}. Then the probability of at least
one shortfall has exponent \(c\), whereas under block independence
the probability of at least \(q\) shortfalls has exponent \(qc\).
The latter conclusion need not hold without independence. Indeed, for
identical blocks with perfectly dependent shortfall events, simultaneous
shortfalls have the same exponent as a shortfall in a single block.

\subsection{Aggregate losses and portfolio capital}

We now consider weighted aggregate claims. Let \(w_j(t)>0\) be
deterministic weights such that
\[
    \frac{\log w_j(t)}{\log t}\longrightarrow\theta_j\in\R,
    \qquad j=1,\ldots,m,
\]
and define
\[
    S_w(t):=\sum_{j=1}^m w_j(t)S_j(t).
\]
The weights may represent currency conversions or proportional retained
shares. Taking \(w_j(t)\equiv1\) for every \(j\) gives the total
unweighted claims. Throughout this subsection, dependence across blocks
is unrestricted.

\begin{proposition}[Aggregate shortfall exponent]
\label{prop:weighted-aggregate}
Let \(K(t)>0\) satisfy \(\log K(t)/\log t\to\kappa\), where
\begin{equation}
\label{eq:thresholdAggregate}
    \kappa>\kappa_0:=
    \max_{1\le j\le m}
    \left\{\theta_j+\rho_j\max\left(1,\frac1{\alpha_j}\right)\right\}.
\end{equation}
Then,
\[
    \Pp(S_w(t)>K(t))=t^{-d(\kappa)+o(1)},
\]
where
\[
    d(\kappa):=\min_{1\le j\le m}
        \{\alpha_j(\kappa-\theta_j)-\rho_j\}.
\]
\end{proposition}

\begin{proof}
Fix \(j\in\{1,\ldots,m\}\) and \(b>0\). By
Lemma~\ref{lem:quantile-tail},
\[
\begin{aligned}
    \frac{\log\bigl(bK(t)/(w_j(t)a_{j,\nu_j(t)})\bigr)}
         {\log\nu_j(t)}
    &=\frac{\log b+\log K(t)-\log w_j(t)}{\log\nu_j(t)}
        -\frac{\log a_{j,\nu_j(t)}}{\log\nu_j(t)} \\
    &\longrightarrow
        \frac{\kappa-\theta_j}{\rho_j}-\frac1{\alpha_j}
        =:x_j(\kappa).
\end{aligned}
\]
Condition~\eqref{eq:thresholdAggregate} is equivalent to
\(x_j(\kappa)>\ell_j\) for all \(j=1,\ldots,m\). Hence, by
Remark~\ref{rem:general-thresholds} and the limit
\(\log\nu_j(t)/\log t\to\rho_j\),
\[
\begin{aligned}
    \Pp(w_j(t)S_j(t)>bK(t))
    &=\nu_j(t)^{-\alpha_jx_j(\kappa)+o(1)} \\
    &=t^{-d_j(\kappa)+o(1)},
\end{aligned}
\]
where
\[
    d_j(\kappa):=\rho_j\alpha_jx_j(\kappa)
        =\alpha_j(\kappa-\theta_j)-\rho_j.
\]
On the other hand, nonnegativity of the claims and positivity of the
weights give
\[
\begin{aligned}
    \max_{1\le j\le m}\Pp(w_j(t)S_j(t)>K(t))
    &\le\Pp(S_w(t)>K(t)) \\
    &\le\sum_{j=1}^m\Pp(w_j(t)S_j(t)>K(t)/m).
\end{aligned}
\]
Applying the preceding marginal estimate with \(b=1\) and \(b=1/m\),
and using that \(m\) is fixed, we obtain
\[
    \Pp(S_w(t)>K(t))=t^{-\min_{1\le j\le m}d_j(\kappa)+o(1)}.
\]
This proves the assertion.
\end{proof}

The aggregate cost \(d(\kappa)\) is the minimum of the affine marginal
costs \(d_j(\kappa)\). Consequently, the block which determines the
exponent may change with the capital growth rate. For fixed positive
weights, \(\theta_j=0\), so that their magnitudes do not enter the
logarithmic exponent, although they may affect finer asymptotics.
Likewise, the aggregate exponent holds under arbitrary dependence,
while joint loss patterns may depend on the coupling of the blocks.

The following corollary gives a capital rule for a prescribed aggregate
shortfall cost.

\begin{corollary}[Portfolio capital calibration]
\label{cor:portfolio-capital-calibration}
Let \(c>0\), and define
\[
    \kappa_c:=\max_{1\le j\le m}
        \left\{\theta_j+\frac{\rho_j+c}{\alpha_j}\right\}.
\]
Suppose that \(\kappa_c>\kappa_0\). Then, for every positive capital
threshold \(K_c(t)\) satisfying
\[
    \frac{\log K_c(t)}{\log t}\longrightarrow\kappa_c,
\]
we have
\[
    \Pp(S_w(t)>K_c(t))=t^{-c+o(1)}.
\]
Moreover, \(\kappa_c\) is the smallest capital growth exponent in the
region \(\kappa>\kappa_0\) for which \(d(\kappa)\ge c\).
\end{corollary}

\begin{proof}
By definition of \(\kappa_c\),
\[
    d_j(\kappa_c)=\alpha_j(\kappa_c-\theta_j)-\rho_j\ge c,
    \qquad j=1,\ldots,m,
\]
with equality for at least one index. Thus \(d(\kappa_c)=c\), and
Proposition~\ref{prop:weighted-aggregate} gives the probability estimate.
Conversely, if \(d(\kappa)\ge c\), then \(d_j(\kappa)\ge c\) for
every \(j\), and therefore
\[
    \kappa\ge\theta_j+\frac{\rho_j+c}{\alpha_j},
    \qquad j=1,\ldots,m.
\]
Taking the maximum over \(j\) yields \(\kappa\ge\kappa_c\).
\end{proof}

The condition \(\kappa_c>\kappa_0\) ensures that the calibration lies
in the region covered by the scalar theorem. When it fails, any choice
\(\kappa>\max\{\kappa_0,\kappa_c\}\) gives an aggregate shortfall
exponent strictly larger than \(c\).

\begin{example}[A change in the block governing aggregate risk]
\label{ex:heterogeneous-capital-crossover}
Consider two unweighted blocks with
\[
    (\alpha_1,\rho_1)=\left(\frac32,\frac14\right),
    \qquad
    (\alpha_2,\rho_2)=(2,1).
\]
For instance, these parameters arise from Pareto claim sizes with the
indicated tail indices, Pareto interarrival times with tail index
\(1/4\) in the first block, and deterministic unit interarrival times
in the second block. Within each block, the claim sizes are assumed
independent of the arrivals. In this case, \(\kappa_0=1\), and
Proposition~\ref{prop:weighted-aggregate} gives, for \(\kappa>1\),
\[
    d(\kappa)=
    \min\left\{\frac32\kappa-\frac14,\,2\kappa-1\right\}.
\]
For \(1<\kappa<3/2\), the minimum is attained by the second block,
whose greater exposure growth outweighs its lighter claim-size tail.
For \(\kappa>3/2\), the minimum is attained by the first block. At
\(\kappa=3/2\), both marginal costs are equal to \(2\).

For a target cost \(c=3/2\), Corollary~\ref{cor:portfolio-capital-calibration}
yields
\[
    \kappa_c=\max\left\{\frac76,\frac54\right\}
        =\frac54>\kappa_0.
\]
Consequently,
\[
    \Pp\left(S_1(t)+S_2(t)>t^{5/4}\right)
    =t^{-3/2+o(1)}.
\]
This estimate holds under arbitrary dependence across the two blocks.
\end{example}

\section{Conclusion}

We have established logarithmic large-deviation asymptotics for
heavy-tailed random sums under an asymmetric count-scale condition that
controls only the logarithmic order of the random number of summands.
The resulting estimates are uncentered and hold whenever the claim-size
survival function satisfies \(\log\bar F(x)/\log x\to-\alpha\) as
\(x\to\infty\), for any \(\alpha>0\). Thus, regular variation of the
claim-size distribution is not required. 

A main contribution is the universal renewal result. We showed that
every  renewal count satisfies the count-scale condition with
the intrinsic scale
\[
        \nu_N(t)
        =
        \bigl(1-\Ee[e^{-\tau_1/t}]\bigr)^{-1},
\]
without any moment, regular-variation, or domain-of-attraction
assumption on the interarrival distribution. Moreover, the mean renewal
count
\[
        m(t)=\Ee[N(t)]
\]
is itself an admissible  count-scale. In particular, the main
large-deviation estimate remains valid with the classical mean-count
scale even when the renewal count does not concentrate relatively
around its mean. The framework therefore covers both finite-mean and
infinite-mean interarrival regimes within the same formulation.

For interarrival distributions with regularly varying tails, the
intrinsic and mean-count scales have the same logarithmic growth
exponent, which yields explicit calendar-time decay rates for the
random-sum tail probabilities.

Finally, we applied the scalar asymptotics to finite heterogeneous
insurance portfolios. Without cross-block independence, finite unions
of block shortfalls and nonnegative weighted aggregate losses are
governed by the cheapest marginal route. Under independent blocks,
simultaneous shortfall costs add, and \(q\)-out-of-\(m\) shortfall
probabilities are governed by the \(q\) smallest marginal costs. These
asymptotic costs also lead to explicit capital-calibration rules and
identify the dominant routes to portfolio shortfall.

\appendix

\section{Proofs}
\label{app:proofs}

\subsection{Preliminary results}

\begin{lemma}
\label{lem:moment}
Let \(X\ge0\) have a survival function \(\bar F\) satisfying
\eqref{eq:polynomial-tail-order} for some \(\alpha>0\). Then
\(\Ee[X^p]<\infty\) for every \(p\in(0,\alpha)\).
\end{lemma}

\begin{proof}
Fix \(p\in(0,\alpha)\). Since \(X\) is nonnegative, we have
\[
    \Ee[X^p]=p\int_0^\infty x^{p-1}\Pp(X>x)\,dx.
\]
Choose \(\varepsilon>0\) such that \(p+\varepsilon<\alpha\). By
\eqref{eq:polynomial-tail-order}, there exists \(x_0>0\) such that
\(\Pp(X>x)\le x^{-\alpha+\varepsilon}\) for all \(x\ge x_0\).
Consequently,
\[
    \Ee[X^p]
    \le x_0^p+p\int_{x_0}^\infty x^{p-1-\alpha+\varepsilon}\,dx
    <\infty,
\]
where the last inequality follows from \(p-\alpha+\varepsilon<0\).
\end{proof}

\begin{lemma}
\label{lem:quantile-tail}
Let \(X\ge0\) have distribution function \(F\) and survival function
\(\bar F\) satisfying \eqref{eq:polynomial-tail-order} for some
\(\alpha>0\). For \(y>1\), set
\[
    a_y:=F^{\leftarrow}\left(1-\frac1y\right).
\]
Then,
\[
    \lim_{y\to\infty}\frac{\log a_y}{\log y}=\frac1\alpha,
    \qquad
    \lim_{y\to\infty}\frac{\log\bar F(a_y)}{\log y}=-1.
\]
Moreover, if \(u_y>0\) satisfies
\[
    \lim_{y\to\infty}
    \frac{\log(u_y/a_y)}{\log y}=\beta\ge0,
\]
then
\[
    \lim_{y\to\infty}\frac{\log u_y}{\log y}
    =\frac1\alpha+\beta,
    \qquad
    \lim_{y\to\infty}\frac{\log\bar F(u_y)}{\log y}
    =-1-\alpha\beta.
\]
\end{lemma}

\begin{proof}
By assumption,
\begin{equation}
\label{eq:limitRV}
    \lim_{x\to\infty}\frac{\log\bar F(x)}{\log x}=-\alpha.
\end{equation}

We first show that \(a_y\to\infty\). Otherwise, there would exist
\(M>0\) and a sequence \(y_n\to\infty\) such that \(a_{y_n}\le M\)
for all \(n\). By monotonicity of \(F\) and the definition of the
generalized inverse,
\[
    F(M)\ge F(a_{y_n})\ge1-\frac1{y_n}.
\]
Letting \(n\to\infty\), we obtain \(F(M)=1\). This implies that
\(\bar F\) vanishes eventually, contrary to \eqref{eq:limitRV}.

Fix \(\varepsilon\in(0,\alpha)\) and \(r>0\). Since \(a_y\to\infty\),
it follows from \eqref{eq:limitRV} that, for all sufficiently large
\(y\),
\[
    a_y^{-(\alpha+\varepsilon)}\le\bar F(a_y),
    \qquad
    \bar F(a_y-r)\le(a_y-r)^{-(\alpha-\varepsilon)}.
\]
On the other hand, by the definition of \(F^{\leftarrow}\),
\[
    \bar F(a_y)\le\frac1y\le\bar F(a_y-r).
\]
Combining these inequalities gives
\[
    a_y^{-(\alpha+\varepsilon)}
    \le\frac1y
    \le(a_y-r)^{-(\alpha-\varepsilon)},
\]
so that
\[
    a_y-r\le y^{1/(\alpha-\varepsilon)},
    \qquad
    y^{1/(\alpha+\varepsilon)}\le a_y.
\]
Consequently,
\[
    y^{1/(\alpha+\varepsilon)}
    \le a_y\le r+y^{1/(\alpha-\varepsilon)}.
\]
Taking logarithms and dividing by \(\log y\), we obtain
\[
    \frac1{\alpha+\varepsilon}
    \le\frac{\log a_y}{\log y}
    \le\frac1{\alpha-\varepsilon}
        +\frac{\log\bigl(1+r\,y^{-1/(\alpha-\varepsilon)}\bigr)}
              {\log y}.
\]
Since the last term tends to zero as \(y\to\infty\), it follows that
\[
    \frac1{\alpha+\varepsilon}
    \le\liminf_{y\to\infty}\frac{\log a_y}{\log y}
    \le\limsup_{y\to\infty}\frac{\log a_y}{\log y}
    \le\frac1{\alpha-\varepsilon}.
\]
Letting \(\varepsilon\downarrow0\), we conclude that
\[
    \lim_{y\to\infty}\frac{\log a_y}{\log y}=\frac1\alpha.
\]
Moreover, since \(a_y\to\infty\), \eqref{eq:limitRV} gives
\[
\begin{aligned}
    \lim_{y\to\infty}\frac{\log\bar F(a_y)}{\log y}
    &=\left(\lim_{y\to\infty}
        \frac{\log\bar F(a_y)}{\log a_y}\right)
        \left(\lim_{y\to\infty}
        \frac{\log a_y}{\log y}\right) \\
    &=(-\alpha)\frac1\alpha=-1.
\end{aligned}
\]

As for the second part, suppose that \(u_y>0\) satisfies
\[
    \lim_{y\to\infty}
    \frac{\log(u_y/a_y)}{\log y}=\beta\ge0.
\]
Then,
\[
\begin{aligned}
    \lim_{y\to\infty}\frac{\log u_y}{\log y}
    &=\lim_{y\to\infty}\frac{\log a_y}{\log y}
        +\lim_{y\to\infty}\frac{\log(u_y/a_y)}{\log y} \\
    &=\frac1\alpha+\beta.
\end{aligned}
\]
In particular, \(u_y\to\infty\). Applying \eqref{eq:limitRV} once more,
we obtain
\[
\begin{aligned}
    \lim_{y\to\infty}\frac{\log\bar F(u_y)}{\log y}
    &=\left(\lim_{y\to\infty}
        \frac{\log\bar F(u_y)}{\log u_y}\right)
        \left(\lim_{y\to\infty}
        \frac{\log u_y}{\log y}\right) \\
    &=-\alpha\left(\frac1\alpha+\beta\right)
        =-1-\alpha\beta.
\end{aligned}
\]
The proof is complete.
\end{proof}

\begin{lemma}
\label{lem:one-minus-power}
Let \(m_n>0\) with \(m_n\to\infty\), and let \(u_n\in[0,1]\) satisfy
\(m_nu_n\to0\). Then,
\[
    1-(1-u_n)^{m_n}=m_nu_n(1+o(1)).
\]
\end{lemma}

\begin{proof}
Let \(m\ge2\) and \(u\in[0,1]\). Since
\[
    1-(1-u)^m=\int_0^u m(1-y)^{m-1}\,dy
\]
and \(y\mapsto(1-y)^{m-1}\) is decreasing on \([0,u]\), we have
\[
    mu(1-u)^{m-1}\le1-(1-u)^m\le mu.
\]
Furthermore, Bernoulli's inequality with exponent \(m-1\ge1\)
gives\footnote{For \(v\ge-1\) and \(r\ge1\), Bernoulli's inequality
states that \((1+v)^r\ge1+rv\); see \cite[p.~34]{Mitrinovic1970}.
Here, we take \(v=-u\) and \(r=m-1\).}
\[
    (1-u)^{m-1}\ge1-(m-1)u.
\]
Consequently,
\[
    mu\bigl[1-(m-1)u\bigr]\le1-(1-u)^m\le mu.
\]
Taking \(m=m_n\) and \(u=u_n\) for all sufficiently large \(n\), the
assertion follows from \(m_nu_n\to0\).
\end{proof}

\begin{lemma}
\label{lem:deterministic-tail}
Let \(X_1,X_2,\ldots\) be iid nonnegative random variables whose common
survival function \(\bar F\) satisfies \eqref{eq:polynomial-tail-order}
for some \(\alpha>0\). Let \(m_n\in\N\) with \(m_n\to\infty\), and set
\(S_m:=X_1+\cdots+X_m\). Suppose that \(u_n>0\) satisfies
\[
    \lim_{n\to\infty}
    \frac{\log(u_n/a_{m_n})}{\log m_n}
    =\gamma>\left(1-\frac1\alpha\right)^+.
\]
Then,
\[
    \lim_{n\to\infty}
    \frac1{\log m_n}\log\Pp(S_{m_n}>u_n)=-\alpha\gamma.
\]
Equivalently,
\[
    \Pp(S_{m_n}>u_n)=m_n^{-\alpha\gamma+o(1)}.
\]
\end{lemma}

\begin{proof}
We first prove the lower bound. By the second part of
Lemma~\ref{lem:quantile-tail},
\[
    \bar F(u_n)=m_n^{-1-\alpha\gamma+o(1)},
\]
and therefore
\[
    m_n\bar F(u_n)=m_n^{-\alpha\gamma+o(1)}\longrightarrow0.
\]
Set \(M_{m_n}:=\max_{1\le k\le m_n}X_k\). By independence and
Lemma~\ref{lem:one-minus-power},
\[
\begin{aligned}
    \Pp(M_{m_n}>u_n)
    &=1-\bigl(1-\bar F(u_n)\bigr)^{m_n} \\
    &=m_n\bar F(u_n)(1+o(1)) \\
    &=m_n^{-\alpha\gamma+o(1)}.
\end{aligned}
\]
Since \(S_{m_n}\ge M_{m_n}\), it follows that
\[
    \liminf_{n\to\infty}
    \frac{\log\Pp(S_{m_n}>u_n)}{\log m_n}
    \ge-\alpha\gamma.
\]

As for the upper bound, fix \(h>\alpha\gamma\) and, for all sufficiently
large \(n\), set
\[
    b_n:=\frac{u_n}{2h\log m_n}.
\]
Then,
\begin{align}
    \{S_{m_n}>u_n\}
    &\subseteq\{M_{m_n}>b_n\}
        \cup\bigl(\{S_{m_n}>u_n\}\cap\{M_{m_n}\le b_n\}\bigr)
        \notag\\
    &\subseteq\{M_{m_n}>b_n\}
        \cup\left\{
        \sum_{k=1}^{m_n}X_k\one_{\{X_k\le b_n\}}>u_n
        \right\}.
\label{eq:sumMax}
\end{align}
Since \(h\) is fixed, we have
\[
\begin{aligned}
    \frac{\log(b_n/a_{m_n})}{\log m_n}
    &=\frac{\log(u_n/a_{m_n})}{\log m_n}
        -\frac{\log(2h\log m_n)}{\log m_n} \\
    &\longrightarrow\gamma.
\end{aligned}
\]
It follows from Lemma~\ref{lem:quantile-tail} that
\[
    m_n\bar F(b_n)=m_n^{-\alpha\gamma+o(1)}\longrightarrow0.
\]
Thus, by independence and Lemma~\ref{lem:one-minus-power},
\[
\begin{aligned}
    \Pp(M_{m_n}>b_n)
    &=1-\bigl(1-\bar F(b_n)\bigr)^{m_n} \\
    &=m_n\bar F(b_n)(1+o(1)) \\
    &=m_n^{-\alpha\gamma+o(1)}.
\end{aligned}
\]

It remains to estimate the truncated sum in \eqref{eq:sumMax}. To that
end, define
\[
    Y_{n,k}:=X_k\one_{\{X_k\le b_n\}},
    \qquad 1\le k\le m_n,
\]
and set \(q:=1/\alpha+\gamma\). The condition on \(\gamma\) implies
\(q^{-1}<\alpha\wedge1\), so that we can choose \(p\) with
\[
    q^{-1}<p<\alpha\wedge1.
\]
Then \(p<1\), \(p<\alpha\), and \(pq>1\). In particular,
\(\Ee[X_1^p]<\infty\) by Lemma~\ref{lem:moment}. Moreover, for
\(r=1,2\),
\[
    Y_{n,k}^r\le b_n^{r-p}X_k^p.
\]
Since \(u_n=m_n^{q+o(1)}\) by Lemma~\ref{lem:quantile-tail}, we obtain
\[
\begin{aligned}
    \frac{m_n\Ee[Y_{n,1}^r]}{b_n^{r-1}u_n}
    &\le\Ee[X_1^p]\frac{m_n b_n^{1-p}}{u_n} \\
    &=\Ee[X_1^p]\,m_nu_n^{-p}
        (2h\log m_n)^{-(1-p)} \\
    &=\Ee[X_1^p]\,m_n^{1-pq+o(1)}
        (2h\log m_n)^{-(1-p)}
        \longrightarrow0,
    \qquad r=1,2,
\end{aligned}
\]
where the convergence follows from \(pq>1\). Consequently, for all
sufficiently large \(n\),
\[
    m_n\Ee[Y_{n,1}]\le\frac{u_n}{2},
    \qquad
    m_n\Ee[Y_{n,1}^2]\le\frac{b_nu_n}{12}.
\]

For each \(n\), the variables \(Y_{n,1},\ldots,Y_{n,m_n}\) are
independent and satisfy \(0\le Y_{n,k}\le b_n\). Applying the one-sided
Bernstein inequality \cite[Proposition~2.14]{wainwright2019highdimensional}
with deviation level \(u_n/2\),\footnote{For independent centered random
variables \(Z_1,\ldots,Z_m\) satisfying \(\lvert Z_k\rvert\le b\) a.s.,
where \(b>0\), the one-sided Bernstein inequality gives
\[
    \Pp\left(\sum_{k=1}^m Z_k\ge s\right)
    \le\exp\left\{
        -\frac{s^2}{2\left(\sum_{k=1}^m\Ee[Z_k^2]+bs/3\right)}
        \right\},
    \qquad s>0.
\]
Here, \(Z_k=Y_{n,k}-\Ee[Y_{n,k}]\), so that
\(\lvert Z_k\rvert\le b_n\) and
\(\sum_{k=1}^{m_n}\Ee[Z_k^2]\le m_n\Ee[Y_{n,1}^2]\).}
we obtain
\[
\begin{aligned}
    \Pp\left(\sum_{k=1}^{m_n}Y_{n,k}>u_n\right)
    &\le\Pp\left(
        \sum_{k=1}^{m_n}\bigl(Y_{n,k}-\Ee[Y_{n,k}]\bigr)
        >\frac{u_n}{2}\right) \\
    &\le\exp\left\{
        -\frac{u_n^2}{8m_n\Ee[Y_{n,1}^2]+\frac43b_nu_n}
        \right\} \\
    &\le\exp\left\{-\frac{u_n}{2b_n}\right\} \\
    &=m_n^{-h}.
\end{aligned}
\]
The penultimate inequality follows from
\[
    8m_n\Ee[Y_{n,1}^2]+\frac43b_nu_n\le2b_nu_n.
\]

Combining these estimates with \eqref{eq:sumMax} gives
\[
    \Pp(S_{m_n}>u_n)
    \le m_n^{-\alpha\gamma+o(1)}+m_n^{-h}.
\]
Since \(h>\alpha\gamma\), we conclude that
\[
    \limsup_{n\to\infty}
    \frac{\log\Pp(S_{m_n}>u_n)}{\log m_n}
    \le-\alpha\gamma.
\]
Together with the lower bound, this proves the assertion.
\end{proof}

\subsection{Proof of the main result}

We now prove Theorem~\ref{thm:random-sum-main} by combining the
deterministic-count estimate in Lemma~\ref{lem:deterministic-tail}
with Assumption~\ref{ass:count-scale}.

\begin{proof}[Proof of Theorem~\ref{thm:random-sum-main}]
Fix \(x>\ell_\alpha\), and set \(u_t:=a_{\nu(t)}\nu(t)^x\). Let
\(\varepsilon\in(0,1)\) be sufficiently small that
\[
    \frac{x-\varepsilon/\alpha}{1+\varepsilon}>\ell_\alpha,
\]
and define
\[
    m_-(t):=\lfloor\nu(t)^{1-\varepsilon}\rfloor,
    \qquad
    m_+(t):=\lceil\nu(t)^{1+\varepsilon}\rceil.
\]

We first prove the lower bound. By independence of the count and the
summands, and since the summands are nonnegative,
\[
    \Pp(S(t)>u_t)
    \ge\Pp(N(t)\ge m_-(t))
        \Pp\left(\sum_{k=1}^{m_-(t)}X_k>u_t\right).
\]
Assumption~\ref{ass:count-scale} gives
\(\Pp(N(t)\ge m_-(t))\to1\), and therefore
\[
    \frac{\log\Pp(N(t)\ge m_-(t))}{\log\nu(t)}\longrightarrow0.
\]
Moreover, since
\[
    \frac{\log m_-(t)}{\log\nu(t)}\longrightarrow1-\varepsilon,
\]
it follows from Lemma~\ref{lem:quantile-tail} that
\[
    \frac{\log(u_t/a_{m_-(t)})}{\log m_-(t)}
    \longrightarrow\frac{x+\varepsilon/\alpha}{1-\varepsilon}
    >\ell_\alpha.
\]
Hence, by Lemma~\ref{lem:deterministic-tail},
\[
    \Pp\left(\sum_{k=1}^{m_-(t)}X_k>u_t\right)
    =m_-(t)^{-\alpha\frac{x+\varepsilon/\alpha}{1-\varepsilon}+o(1)}.
\]
Using again \(\log m_-(t)/\log\nu(t)\to1-\varepsilon\), we obtain
\[
\begin{aligned}
    \Pp\left(\sum_{k=1}^{m_-(t)}X_k>u_t\right)
    &=\nu(t)^{-\alpha(x+\varepsilon/\alpha)+o(1)} \\
    &=\nu(t)^{-\alpha x-\varepsilon+o(1)}.
\end{aligned}
\]
Consequently,
\[
    \liminf_{t\to\infty}
    \frac{\log\Pp(S(t)>u_t)}{\log\nu(t)}
    \ge-\alpha x-\varepsilon.
\]

As for the upper bound, nonnegativity of the summands gives
\[
    \Pp(S(t)>u_t)
    \le\Pp(N(t)>m_+(t))
        +\Pp\left(\sum_{k=1}^{m_+(t)}X_k>u_t\right).
\]
Choose \(h>\alpha x+2\). By Assumption~\ref{ass:count-scale},
\[
    \Pp(N(t)>m_+(t))=o\bigl(\nu(t)^{-h}\bigr).
\]
On the other hand, \(\log m_+(t)/\log\nu(t)\to1+\varepsilon\), so that
Lemma~\ref{lem:quantile-tail} gives
\[
    \frac{\log(u_t/a_{m_+(t)})}{\log m_+(t)}
    \longrightarrow\frac{x-\varepsilon/\alpha}{1+\varepsilon}
    >\ell_\alpha.
\]
Applying Lemma~\ref{lem:deterministic-tail}, we obtain
\[
\begin{aligned}
    \Pp\left(\sum_{k=1}^{m_+(t)}X_k>u_t\right)
    &=m_+(t)^{-\alpha\frac{x-\varepsilon/\alpha}{1+\varepsilon}+o(1)} \\
    &=\nu(t)^{-\alpha x+\varepsilon+o(1)}.
\end{aligned}
\]
It follows that
\[
    \limsup_{t\to\infty}
    \frac{\log\Pp(S(t)>u_t)}{\log\nu(t)}
    \le-\alpha x+\varepsilon.
\]
Letting \(\varepsilon\downarrow0\) in the lower and upper bounds yields
\[
    \lim_{t\to\infty}
    \frac{\log\Pp(S(t)>u_t)}{\log\nu(t)}=-\alpha x.
\]
Finally, if \(\log\nu(t)/\log t\to\rho\in(0,\infty)\), then
\[
\begin{aligned}
    \lim_{t\to\infty}\frac{\log\Pp(S(t)>u_t)}{\log t}
    &=\left(\lim_{t\to\infty}
        \frac{\log\Pp(S(t)>u_t)}{\log\nu(t)}\right)
        \left(\lim_{t\to\infty}
        \frac{\log\nu(t)}{\log t}\right) \\
    &=-\rho\alpha x.
\end{aligned}
\]
The proof is complete.
\end{proof}

\end{document}